\pdfoutput=1
\RequirePackage{ifpdf}
\ifpdf 
\documentclass[pdftex]{sigma}
\else
\documentclass{sigma}
\fi

\numberwithin{equation}{section}
\newtheorem{Theorem}{Theorem}[section]
\newtheorem{Corollary}[Theorem]{Corollary}
\newtheorem{Proposition}[Theorem]{Proposition}

\begin{document}

\newcommand{\arXivNumber}{1804.02856}

\renewcommand{\thefootnote}{}

\renewcommand{\PaperNumber}{088}

\FirstPageHeading

\ShortArticleName{Discrete Orthogonal Polynomials with Hypergeometric Weights and Painlev\'e VI}

\ArticleName{Discrete Orthogonal Polynomials\\ with Hypergeometric Weights and Painlev\'e VI\footnote{This paper is a~contribution to the Special Issue on Painlev\'e Equations and Applications in Memory of Andrei Kapaev. The full collection is available at \href{https://www.emis.de/journals/SIGMA/Kapaev.html}{https://www.emis.de/journals/SIGMA/Kapaev.html}}}

\Author{Galina FILIPUK~$^\dag$ and Walter VAN ASSCHE~$^\ddag$}
\AuthorNameForHeading{G.~Filipuk and W.~Van Assche}

\Address{$^\dag$~Faculty of Mathematics, Informatics and Mechanics, University of Warsaw,\\
\hphantom{$^\dag$}~Banacha 2, Warsaw, 02-097, Poland}
\EmailD{\href{mailto:filipuk@mimuw.edu.pl}{filipuk@mimuw.edu.pl}}

\Address{$^\ddag$~Department of Mathematics, KU Leuven,\\
\hphantom{$^\ddag$}~Celestijnenlaan 200B box 2400, BE-3001 Leuven, Belgium}
\EmailD{\href{mailto:walter.vanassche@kuleuven.be}{walter.vanassche@kuleuven.be}}

\ArticleDates{Received April 10, 2018, in final form August 20, 2018; Published online August 24, 2018}

\Abstract{We investigate the recurrence coefficients of discrete orthogonal polynomials on the non-negative integers with hypergeometric weights
and show that they satisfy a system of non-linear difference equations and a non-linear second order differential equation in one of the parameters of the weights. The non-linear difference equations form a pair of discrete Painlev\'e equations and the differential equation is the $\sigma$-form of the sixth Painlev\'e equation. We briefly investigate the asymptotic behavior of the recurrence coefficients as $n\to \infty$ using the discrete Painlev\'e equations.}

\Classification{33C45; 33E17; 34M55; 42C05}

\renewcommand{\thefootnote}{\arabic{footnote}}
\setcounter{footnote}{0}

\section{Introduction}

In the past few years many semi-classical orthogonal polynomials were investigated and discrete and continuous Painlev\'e equations were found for their recurrence coefficients. In this paper we are interested in some discrete orthogonal polynomials on the integers $\mathbb{N}=\{0,1,2,3,\ldots\}$. For the orthonormal polynomials one has
\begin{gather*} \sum_{k=0}^\infty p_n(k)p_m(k) w_k = \delta_{m,n}, \qquad p_n(x) = \gamma_n x^n +\cdots \end{gather*}
and the three term recurrence relation is
\begin{gather*} xp_n(x) = a_{n+1}p_{n+1}(x) + b_n p_n(x) + a_n p_{n-1}(x), \end{gather*}
where, as usual, we take $a_0=0$, and the weights are such that all the moments are finite
\begin{gather*} m_n = \sum_{k=0}^\infty k^n w_k < \infty , \qquad n=0,1,2,\ldots. \end{gather*}
For the monic orthogonal polynomials $P_n = p_n/\gamma_n$ the recurrence relation becomes
\begin{gather*} xP_n(x) = P_{n+1}(x) + b_n P_n(x) + a_n^2 P_{n-1}(x). \end{gather*}
The following families have already been analyzed earlier:
\begin{itemize}\itemsep=0pt
 \item The Charlier polynomials $C_n(x;a)$ $(a >0)$ \cite[Section~VI.1]{Chihara}, \cite[Section~9.14]{Koekoek} form a~system of classical orthogonal polynomials on the integers $\mathbb{N}$ satisfying
 \begin{gather*} \sum_{k=0}^\infty C_n(k;a)C_m(k;a) \frac{a^k}{k!} = 0, \qquad n \neq m. \end{gather*}
 The monic Charlier polynomials $P_n(x) = (-1)^n a^n C_n(x;a)$ satisfy the recurrence relation
 \begin{gather*} xP_n(x) = P_{n+1}(x) + (n+a) P_n(x) + na P_{n-1}(x), \end{gather*}
 hence the weights are $w_k = a^k/k!$ and the recurrence coefficients are $a_n^2=na$ and $b_n = n+a$. These are simple polynomial expressions in $n$ and $a$.
 \item The Meixner polynomials $M_n(x;\beta,c)$ $(\beta >0$, $0 < c < 1)$ \cite[Section~VI.3]{Chihara}, \cite[Section~9.10]{Koekoek} are also a family of classical orthogonal polynomials:
 \begin{gather*} \sum_{k=0}^\infty M_n(k;\beta,c) M_m(k;\beta,c) \frac{(\beta)_k c^k}{k!} = 0, \qquad n \neq m, \end{gather*}
 and the recurrence coefficients are again simple and given by
 \begin{gather} \label{Meixner}
 a_n^2 = \frac{n(n+\beta-1) c}{(1-c)^2}, \qquad b_n = \frac{n+(n+\beta)c}{1-c} .
 \end{gather}
When $\beta=-N$ is a negative integer and $c=\frac{p}{p-1}$, with $0 < p < 1$, then one finds Krawtchouk polynomials $K_n(x;p,N)$. This is a finite family of polynomials which are orthogonal for the binomial distribution. The recurrence coefficients are
\begin{gather*} a_n^2 = np(1-p)(N+1-n), \qquad b_n = p(N-n)+n(1-p). \end{gather*}
Note that $a_{N+1}^2=0$, which comes from the fact that this is a finite family of orthogonal polynomials with a measure supported on $N+1$ points.
 \item Generalized Charlier polynomials with weights $w_k = \frac{a^k}{k! (\beta)_k}$ $(a >0$, $\beta >0)$ were, for $\beta=1$, first considered in \cite{Hounkonnou} and analyzed in~\cite{Mama}. The general case $\beta >0$ was investigated in~\cite{Smet} where the discrete Painlev\'e equations are given, and~\cite{Filipuk2013} where the Painlev\'e differential equation was given. Clarkson~\cite{Clarkson2013} found the connection with the Painlev\'e equation in a~different way, starting from the Hankel determinants and the special function solutions of Painlev\'e equations. For $\beta=1$ the recurrence coefficients are given by $a_n^2 = a\big(1-c_n^2\big)$ and $b_n= n+\sqrt{a} c_n c_{n+1}$, where $c_n$ satisfies the discrete Painlev\'e~II equation
 \begin{gather*} c_{n+1} + c_{n-1} = \frac{n c_n}{\sqrt{a} (1-c_n^2)}, \end{gather*}
 with initial conditions $c_0=1$ and $c_1 = I_1\big(2\sqrt{a}\big)/I_0\big(2\sqrt{a}\big)$, where $I_\nu$ is the modified Bessel function. For $\beta \neq 1$ the recurrence coefficients satisfy
 \begin{gather*}
 \big(a_{n+1}^2-a\big) \big(a_n^2-a\big) = a (b_n-n) (b_n-n+\beta-1), \\
 b_n + b_{n-1} -n + \beta = \frac{an}{a_n^2},
 \end{gather*}
 with initial conditions $a_0^2=0$ and $b_0= \sqrt{a} I_{\beta}\big(2\sqrt{a}\big)/I_{\beta-1}\big(2\sqrt{a}\big)$. This is a limiting case of a discrete Painlev\'e IV equation with surface/symmetry $D_4^{(1)}$ \cite[Section~8.1.16]{KNY}. In \cite[Theorem~2.1]{Filipuk2013} it was also shown that~$b_n$, as a function of the parameter~$a$, satisfies a Painlev\'e V equation with parameter $\delta=0$. Such a Painlev\'e equation can be transformed to a Painlev\'e III equation.
\item Generalized Meixner polynomials with weights $w_k = \frac{(\gamma)_k a^k}{k! (\beta)_k}$ $(a >0$, $\beta>0$, $\gamma >0)$ were for $\beta=1$ investigated in \cite{Boelen2011} and for general $\beta >0$ in~\cite{Smet}. The special case $\beta=\gamma$ gives the Charlier polynomials. In \cite[Theorem~3.1]{Smet} it was shown that the recurrence coefficients are given by $a_n^2 = na - (\gamma-1)u_n$, $b_n=n+\gamma-\beta+a-(\gamma-1)v_n/a$, where $(u_n,v_n)_{n \in \mathbb{N}}$ satisfy the system
\begin{gather*}
 (u_n+v_n)(u_{n+1}+v_n) = \frac{\gamma-1}{a^2} v_n(v_n-a) \left( v_n - a \frac{\gamma-\beta}{\gamma-1} \right), \\
 (u_n+v_n)(u_n+v_{n-1}) = \frac{u_n}{u_n-\frac{an}{\gamma-1}} (u_n+a)\left(u_n+a \frac{\gamma-\beta}{\gamma-1} \right),
\end{gather*}
with initial conditions $u_0=0$ and
\begin{gather*} v_0=\frac{a}{\gamma-1} \left( \gamma-\beta+a- \frac{\gamma a M(\gamma+1,\beta+1,a)}{\beta M(\gamma,\beta,a)} \right), \end{gather*}
where $M(a,b,z)$ is the confluent hypergeometric function. This system of non-linear recurrence relations is a limiting case of the
asymmetric discrete Painlev\'e~IV equation related to d-P$\big(E_6^{(1)}/A_2^{(1)}\big)$ in \cite[Section~8.1.15]{KNY}. In~\cite{Filipuk2011} and~\cite{Boelen2011} it was shown that $v_n$, as a function of $a$, satisfies a Painlev\'e equation. See also~\cite{Clarkson2013} for a more direct approach. If we define a~function $y_n(a)$ by
\begin{gather*} v_n = \frac{a(ay_n' - (1+\beta-2)y_n^2 + (n+1-a+\beta-2\gamma)y_n-n)}{2(\gamma-1)(y_n-1)y_n} , \end{gather*}
then $y_n$ satisfies Painlev\'e~V
\begin{gather*} y_n'' = \left( \frac{1}{2y_n} + \frac{1}{y_n-1} \right) (y_n')^2 - \frac{y_n'}{a} + \frac{(y_n-1)^2}{a^2} \left( Ay_n+\frac{B}{y_n} \right) + \frac{Cy_n}{a} + \frac{Dy_n(y_n+1)}{y_n-1}, \end{gather*}
 with
\begin{gather*} A= \frac{(\beta-1)^2}{2}, \qquad B = - \frac{n^2}{2}, \qquad C= n-\beta+2\gamma, \qquad D = - \frac12. \end{gather*}
When $\gamma=-N$ is a negative integer one deals with generalized Krawtchouk polynomials, which were investigated in \cite{Boelen2013}.
\end{itemize}

All these families of discrete orthogonal polynomials are orthogonal on the integers $\mathbb{N} = \{0,1,2,\ldots \}$. They were studied by Dominici and Marcell\'an in~\cite{DomMarc} who were investigating discrete semi-classical orthogonal polynomials of class one, which also includes the Hahn polynomials (which are orthogonal on a finite set $\{0,1,2.\ldots,N\}$). They gave limit relations between these and other families of orthogonal polynomials. The generalized Charlier polynomials and the generalized Meixner polynomials can also be made orthogonal on the shifted lattice $\mathbb{N} + \beta-1$ if $\beta <2$, and the corresponding recurrence coefficients satisfy the same Painlev\'e equations, but with a different initial value for~$b_0$. The more general setting is to consider the generalized Charlier and Meixner polynomials as orthogonal polynomials on the union of $\mathbb{N}$ and $\mathbb{N} + \beta-1$. See~\cite{Smet} for more details. This general setting on the bi-lattice gives solutions of the Painlev\'e equations depending on a seed function (the moment~$m_0$) that consists of a linear combination of two solutions of the Bessel equation or the Kummer equation.

There is one case of discrete orthogonal polynomials that has not been considered in much detail and which also has recurrence coefficients that satisfy discrete and continuous Painlev\'e equations. Take the weights
\begin{gather} \label{weights}
 w_k = \frac{(\alpha)_k (\beta)_k}{(\gamma)_k k!} c^k, \qquad \alpha,\beta,\gamma >0,\qquad 0 < c < 1,
\end{gather}
which corresponds to case~7 in \cite{DomMarc}. The initial moment of this weight is
\begin{gather*} m_0 = \sum_{k=0}^\infty \frac{(\alpha)_k (\beta)_k}{(\gamma)_k k!} c^k = {}_2F_1(\alpha,\beta;\gamma;c) \end{gather*}
involving the Gauss hypergeometric function, and all the other moments are
\begin{gather*} m_n = \sum_{k=0}^\infty k^n \frac{(\alpha)_k (\beta)_k}{(\gamma)_k k!} c^k = \left( c \frac{{\rm d}}{{\rm d}c} \right)^n m_0. \end{gather*}
So therefore one may expect that the recurrence coefficients of the corresponding orthogonal polynomials satisfy a Painlev\'e VI equation, because Painlev\'e VI has special function solutions in terms of hypergeometric functions, see \cite[Section~7.5]{ClarksonLN}. There are also special function solutions for discrete Painlev\'e equations, and for a review we refer to~\cite{Ramani}. In this paper we will find a system of two first order recurrence relations (see Theorem~\ref{thm1}) which allows us to deduce some asymptotic behavior as $n \to \infty$ (Section~\ref{sec4}). In Section~\ref{sec6} we make the connection with the $\sigma$-form of the sixth Painlev\'e equation (see Theorem \ref{thm2}). Note that Dominici already obtained non-linear recurrence relations for the recurrence coefficients in \cite[Theorem~4]{Domin} which he calls the Laguerre--Freud equations. These are however of higher order than two and neither they are identified as discrete Painlev\'e equations, nor is a connection made with Painlev\'e~VI.

There are also some examples of continuous weights for which the recurrence coefficients of the orthogonal polynomials are related to Painlev\'e~VI. Dai and Zhang~\cite{DaiZhang} and Lyu and Chen~\cite{LyuChen} considered a generalization of the Jacobi weight function on~$[0,1]$,
\begin{gather*} w(x,t) = x^\alpha(1-x)^\beta |x-t|^\gamma, \qquad x \in [0,1], \end{gather*}
where $t$ is a real parameter, and found that Painlev\'e VI is the relevant equation for the recurrence coefficients as a function of $t$.
Chen and Zhang \cite{ChenZhang} investigated another modification of the Jacobi weight on $[0,1]$,
\begin{gather*} w(x,t) = x^\alpha (1-x)^\beta \bigl(A+ B \Theta(x-t)\bigr), \qquad x \in [0,1], \end{gather*}
where $\Theta$ is the Heaviside function, and showed that Painlev\'e VI is appearing for the recurrence coefficients of the corresponding orthogonal polynomials. In both cases the moments can be expressed in terms of the Gauss hypergeometric function, and it is the special function
solution of Painlev\'e VI which is needed to find the recurrence coefficients. See also \cite[Section~6.2.5]{WVA} for this connection.

\section{Hypergeometric weights} \label{sec2}
We will investigate the orthogonal polynomials given by the discrete orthogonality relations
\begin{gather*} \sum_{k=0}^\infty p_n(k)p_m(k) \frac{(\alpha)_k (\beta)_k}{(\gamma)_k k!} c^k = \delta_{m,n}, \end{gather*}
with $\alpha,\beta,\gamma >0$ and $0 < c < 1$, and in particular we want to find the recurrence coefficients $(a_n,b_n)_{n \in \mathbb{N}}$ in the three term recurrence relation
\begin{gather} \label{3TRR}
 xp_n(x) = a_{n+1} p_{n+1}(x) + b_n p_n(x) + a_n p_{n-1}(x) .
\end{gather}
The weights $w_k=w(k)$ can be given as the values at the integers $k \in \mathbb{N}$ of the function
\begin{gather} \label{hyperw}
 w(x) = \frac{\Gamma(\gamma)}{\Gamma(\alpha)\Gamma(\beta)} \frac{\Gamma(\alpha+x) \Gamma(\beta+x)}{\Gamma(\gamma+x) \Gamma(x+1)} c^x.
\end{gather}
Observe that for $\alpha=\gamma$ one finds the weights for the Meixner polynomials, and for $c\to 0$ and $\alpha \to \infty$, in such a way that
$\alpha c \to a >0$, one finds the generalized Meixner weight, which in turn for $\beta=\gamma$ gives the Charlier weight and for $a \to 0$ and
$\beta \to \infty$ in such a way that $\beta a \to \hat{a} >0$, gives the generalized Charlier weight. We will use the theory of ladder operators for discrete orthogonal polynomials~\cite{INS}, \cite[Section~6.3]{Ismail}. This uses a discrete potential
\begin{gather*} u(x) = - \frac{w(x)-w(x-1)}{w(x)} = -1 + \frac{(\gamma+x-1) x}{c (\alpha+x-1)(\beta+x-1)}, \end{gather*}
which is rational, with simple poles at $x=-\alpha$ and $x=-\beta$.
Define
\begin{gather*} A_n(x) = a_n \sum_{k=0}^\infty p_n(k)p_n(k-1) \frac{u(x+1)-u(k)}{x+1-k} w_k, \end{gather*}
and
\begin{gather*} B_n(x) = a_n \sum_{k=0}^\infty p_n(k)p_{n-1}(k-1) \frac{u(x+1)-u(k)}{x+1-k} w_k, \end{gather*}
then one has the structure relation
\begin{gather} \label{SR}
 \Delta p_n(x) = A_n(x) p_{n-1}(x) - B_n(x) p_n(x),
\end{gather}
where $\Delta p_n(x) = p_n(x+1)-p_n(x)$ is the forward difference of $p_n(x)$. Some straightforward calculus shows that
\begin{gather*} \frac{u(x+1)-u(k)}{x+1-k} = \frac{c_k}{x+\alpha} + \frac{d_k}{x+\beta}, \end{gather*}
for certain sequences $(c_k,d_k)$, so that
\begin{gather*} \frac{A_n(x)}{a_n} = \frac{u_n}{x+\alpha} + \frac{v_n}{x+\beta}, \qquad B_n(x) = \frac{r_n}{x+\alpha} + \frac{s_n}{x+\beta}, \end{gather*}
where $(u_n,v_n)$ and $(r_n,s_n)$ are sequences depending on $\alpha$, $\beta$, $\gamma$, $c$. The compatibility bet\-ween~\eqref{3TRR} and~\eqref{SR}
gives the relations
\begin{gather*} B_{n+1}(x) + B_n(x) = \frac{x-b_n}{a_n} A_n(x) - u(x+1) + \sum_{j=0}^n \frac{A_j(x)}{a_j} \end{gather*}
and
\begin{gather*} a_{n+1}A_{n+1}(x) - a_n^2 \frac{A_{n-1}(x)}{a_{n-1}} = (x-b_n) B_{n+1}(x) - (x-b_n+1) B_n(x) + 1. \end{gather*}
In our case this gives
\begin{gather}
 \frac{r_n}{x+\alpha} + \frac{s_n}{x+\beta} + \frac{r_{n+1}}{x+\alpha} + \frac{s_{n+1}}{x+\beta} \nonumber\\
\qquad{} = (x-b_n) \left( \frac{u_n}{x+\alpha} + \frac{v_n}{x+\beta} \right) + 1 - \frac{(\gamma+x)(x+1)}{c(\alpha+x)(\beta+x)}
 + \sum_{j=0}^n \left( \frac{u_j}{x+\alpha} + \frac{v_j}{x+\beta} \right),\label{eq1}
\end{gather}
and
\begin{gather}
 a_{n+1}^2 \left( \frac{u_{n+1}}{x+\alpha} + \frac{v_{n+1}}{x+\beta} \right) - a_n^2 \left( \frac{u_{n-1}}{x+\alpha} + \frac{v_{n-1}}{x+\alpha} \right) \nonumber\\
\qquad{} = (x-b_n) \left( \frac{r_{n+1}}{x+\alpha} + \frac{s_{n+1}}{x+\beta} \right) - (x-b_n+1) \left( \frac{r_n}{x+\alpha} + \frac{s_n}{x+\beta} \right) + 1.\label{eq2}
\end{gather}
Our goal is to determine the unknown sequences $a_n$, $b_n$, $u_n$, $v_n$, $r_n$, $s_n$ from the compatibility relations \eqref{eq1}--\eqref{eq2}.

\begin{Proposition} \label{prop1} For $\alpha \neq \beta$ the sequences $u_n$, $v_n$, $r_n$, $s_n$ are given by
\begin{gather}
 (\alpha-\beta) u_n = 2n+1 - \frac{1-c}{c} b_n + \frac{\alpha+\beta-\gamma-1}{c} + (n+1-\beta) \frac{1-c}{c}, \label{un} \\
 (\beta-\alpha) v_n = 2n+1 - \frac{1-c}{c} b_n + \frac{\alpha+\beta-\gamma-1}{c} + (n+1-\alpha) \frac{1-c}{c}, \label{vn}
\end{gather}
and
\begin{gather}
 (\alpha-\beta) r_n= \frac{n(n-1)}{2} - \frac{1-c}{c} a_n^2 + \beta n + \sum_{k=0}^{n-1} b_k, \label{rn} \\
 (\beta-\alpha) s_n= \frac{n(n-1)}{2} - \frac{1-c}{c} a_n^2 + \alpha n + \sum_{k=0}^{n-1} b_k. \label{sn}
\end{gather}
\end{Proposition}
Observe the symmetry $v_n(\alpha,\beta,\gamma,c) = u_n(\beta,\alpha,\gamma,c)$ and $s_n(\alpha,\beta,\gamma,c) = r_n(\beta,\alpha,\gamma,c)$,
which holds because~$a_n$ and~$b_n$ are invariant if we interchange $\alpha$ and $\beta$ (interchanging~$\alpha$ and~$\beta$ leaves the weights~$w_k$ unchanged).

\begin{proof}The identity \eqref{eq1} gives three equations by looking at what happens for $x\to \infty$, $-\alpha$, and $-\beta$. If we let $x \to \infty$ in~\eqref{eq1} then
\begin{gather} \label{eq11}
 u_n + v_n = \frac{1-c}{c} ,
\end{gather}
the residue at $x=-\alpha$ gives
\begin{gather} \label{eq12}
 r_n + r_{n+1} = - u_n(\alpha+b_n) - \frac{(\gamma-\alpha)(1-\alpha)}{c(\beta-\alpha)} + \sum_{j=0}^n u_j,
\end{gather}
and the residue at $x=-\beta$ gives
\begin{gather} \label{eq13}
 s_n + s_{n+1} = -v_n(\beta+b_n) - \frac{(\gamma-\beta)(1-\beta)}{c(\alpha-\beta)} + \sum_{j=0}^n v_j.
\end{gather}
In a similar way we get three equations from \eqref{eq2}: first let $x \to \infty$ to find
\begin{equation*}
 (r_{k+1}+s_{k+1}) - (r_k+s_k) = -1,
\end{equation*}
which after summation (and using $r_0+s_0=0$, which follows because $a_0=0$) gives
\begin{gather} \label{eq21}
 r_n + s_n = -n.
\end{gather}
The residue at $x=-\alpha$ for \eqref{eq2} gives
\begin{gather} \label{eq22}
 a_{n+1}^2 u_{n+1} - a_n^2 u_{n-1} = -r_{n+1} (\alpha+b_n) + r_n (\alpha +b_n-1),
\end{gather}
and the residue at $x=-\beta$ gives
\begin{gather} \label{eq23}
 a_{n+1}^2 v_{n+1} - a_n^2 v_{n-1} = -s_{n+1} (\beta+b_n) + s_n (\beta+b_n-1).
\end{gather}
Adding \eqref{eq12} and \eqref{eq13} while using \eqref{eq11} and \eqref{eq21} gives
\begin{gather} \label{eq31}
 \alpha u_n + \beta v_n = 2n+1 - \frac{1-c}{c}b_n + \frac{\alpha+\beta-\gamma-1}{c} + (n+1) \frac{1-c}{c}.
\end{gather}
Now we can solve the linear system \eqref{eq11} and \eqref{eq31} for $(u_n,v_n)$ and this gives the required expressions \eqref{un}--\eqref{vn}.
In a similar way, we add \eqref{eq22} and \eqref{eq23}, which together with \eqref{eq11} and \eqref{eq21} gives
\begin{equation*}
 \big(a_{n+1}^2 -a_n^2\big) \frac{1-c}{c} = b_n + n - (\alpha r_{n+1}+\beta s_{n+1}) + (\alpha r_n + \beta s_n).
\end{equation*}
Summing this then gives (taking into account that $a_0=0$)
\begin{gather} \label{eq32}
 \alpha r_n + \beta s_n = \frac{n(n-1)}{2} - \frac{1-c}{c} a_n^2 + \sum_{k=0}^{n-1} b_k.
\end{gather}
Now we can solve the linear system \eqref{eq21} and \eqref{eq32} for the unknowns $(r_n,s_n)$ to find the expressions \eqref{rn}--\eqref{sn}.
\end{proof}

\begin{Corollary} \label{cor1}The recurrence coefficients $\big(a_n^2,b_n\big)_{n \in \mathbb{N}}$ are given in terms of $(u_n,r_n)_{n \in \mathbb{N}}$ by
\begin{gather} \label{bnur}
b_n = \frac{n+\alpha-\gamma + (n+\beta)c}{1-c} - (\alpha-\beta) \frac{c}{1-c} u_n,\\
\label{anur} a_n^2 = \frac{n(n+\alpha+\beta-\gamma-1)c}{(1-c)^2} - (\alpha-\beta) \frac{c}{1-c} \left( \frac{c}{1-c} \sum_{j=0}^{n-1} u_j + r_n \right).
\end{gather}
\end{Corollary}

\begin{proof}The formula \eqref{bnur} follows immediately from \eqref{un}. Summing \eqref{bnur} gives
\begin{gather*} \beta n + \sum_{k=0}^{n-1} b_k = \frac{n(n-1)}2 \frac{1+c}{1-c} + \frac{n(\alpha+\beta-\gamma)}{1-c} - (\alpha-\beta) \frac{c}{1-c} \sum_{j=0}^{n-1} u_j, \end{gather*}
and if we use this in \eqref{rn}, then we find
\begin{gather*} \frac{1-c}{c} a_n^2 = \frac{n(n+\alpha+\beta-\gamma-1)}{1-c} - (\alpha-\beta) \left( \frac{c}{1-c} \sum_{j=0}^{n-1} u_j + r_n \right), \end{gather*}
from which \eqref{anur} follows immediately.
\end{proof}

Note that for $\alpha=\gamma$ the weights $w_k$ become Meixner weights, and if we use the recurrence coefficients in~\eqref{Meixner}, then one finds that $u_n=0=r_n$ for all $n \in \mathbb{N}$. The restriction that $\alpha \neq \beta$ in Proposition~\ref{prop1} is not needed but is an artifact of our choice of taking $u_n$ and $r_n$ as the basic sequences. In fact, when $\alpha=\beta$ the discrete potential $u$ has a double pole, resulting in a double pole for~$A_n$ and~$B_n$ as well. In the next section we will use new variables~$x_n$ and~$y_n$ for which the case $\alpha=\beta$ needs no separate analysis.

\section{New variables} \label{sec3}
Recall that
\begin{gather*} \frac{A_n}{a_n} = \frac{u_n}{x+\alpha} + \frac{v_n}{x+\beta} = \frac{(u_n+v_n)x+\beta u_n + \alpha v_n}{(x+\alpha)(x+\beta)} , \end{gather*}
and
\begin{gather*} B_n = \frac{r_n}{x+\alpha} + \frac{s_n}{x+\beta} = \frac{(r_n+s_n)x+\beta r_n + \alpha s_n}{(x+\alpha)(x+\beta)} . \end{gather*}
We already found that $u_n+v_n = \frac{1-c}{c}$, see \eqref{eq11}, and $r_n+s_n= -n$, see~\eqref{eq21}, so we now take new variables
\begin{gather*} x_n=\frac{c}{1-c}(\beta u_n+ \alpha v_n) , \qquad y_n = \beta r_n+\alpha s_n , \end{gather*}
and thus use
\begin{gather} \label{ABxy}
 \frac{A_n}{a_n} = \frac{1-c}{c} \frac{x + x_n}{(x+\alpha)(x+\beta)}, \qquad B_n = \frac{-nx+y_n}{(x+\alpha)(x+\beta)}.
\end{gather}
The advantage of using the unknowns $(x_n,y_n)$ is that these are symmetric in $\alpha$ and $\beta$: they remain unchanged if one interchanges $\alpha$ and $\beta$. We are also going to use one more compatibility relation between the $A_n$ and $B_n$. We already know
\begin{gather} \label{BBA}
 B_{n+1}(x)+B_n(x) = (x-b_n) \frac{A_n(x)}{a_n} - u(x+1) + \sum_{k=0}^n \frac{A_k(x)}{a_k},
\end{gather}
and
\begin{gather} \label{AAB}
 a_{n+1}^2 \frac{A_{n+1}(x)}{a_{n+1}} - a_n^2 \frac{A_{n-1}(x)}{a_{n-1}} = (x-b_n)\bigl(B_{n+1}(x)-B_n(x)\bigr) + 1 - B_n(x).
\end{gather}
Multiply \eqref{AAB} by $A_n/a_n$, then
\begin{gather*} a_{n+1}^2 \frac{A_{n+1}A_n}{a_{n+1}a_n} - a_n^2 \frac{A_nA_{n-1}}{a_na_{n-1}} = (x-b_n) \frac{A_n}{a_n} ( B_{n+1}-B_n ) + (1-B_n) \frac{A_n}{a_n}. \end{gather*}
Replace $(x-b_n)A_n/a_n$ by using \eqref{BBA}, then
\begin{gather*}
 a_{n+1}^2 \frac{A_{n+1}A_n}{a_{n+1}a_n} - a_n^2 \frac{A_nA_{n-1}}{a_na_{n-1}}\\
 \qquad{} = B_{n+1}^2 - B_n^2 +\ (B_{n+1}-B_n) \left( u(x+1) - \sum_{k=0}^n \frac{A_k}{a_k} \right) + (1-B_n) \frac{A_n}{a_n} .
\end{gather*}
Summing from $0$ to $n-1$, taking into account that $A_{-1}=0$ and $B_0=0$, gives
\begin{gather*} a_n^2 \frac{A_nA_{n-1}}{a_na_{n-1}} = B_n^2 + B_n u(x+1) - \sum_{k=0}^{n-1} (B_{k+1}-B_k) \sum_{j=0}^k \frac{A_j}{a_j} + \sum_{k=0}^{n-1} \frac{A_k}{a_k} - \sum_{k=0}^{n-1} B_k \frac{A_k}{a_k} . \end{gather*}
Use summation by parts (for $f_0=0$)
\begin{gather} \label{sumpart}
 \sum_{k=0}^{n-1} (g_{k+1}-g_k) f_k = g_nf_{n-1} - \sum_{k=1}^{n-1} g_k(f_k-f_{k-1})
\end{gather}
to find our third compatibility relation
\begin{gather} \label{AABB}
 a_n^2 \frac{A_n(x)A_{n-1}(x)}{a_na_{n-1}} = B_n^2(x) + u(x+1)B_n(x) + \bigl( 1-B_n(x) \bigr) \sum_{k=0}^{n-1} \frac{A_k(x)}{a_k} .
\end{gather}
Compare this with \cite[equation~(4.8)]{WVA} for the ladder operators corresponding to the differential operator.

Now let us find some relations for the unknown sequences $\big(a_n^2,b_n\big)_n$ and $(x_n,y_n)_n$. If we use~\eqref{ABxy} in~\eqref{BBA} and multiply everything by $(x+\alpha)(x+\beta)$, then we find a quadratic expression in $x$ for which the quadratic term vanishes, so that we only have a~linear term in $x$ and a constant term. The coefficient of $x$ gives the identity
\begin{gather} \label{bx}
 b_n = x_n + \frac{n+(n+ \alpha+\beta)c-\gamma}{1-c},
\end{gather}
which corresponds to \eqref{bnur} in Corollary \ref{cor1}. The constant term gives
\begin{gather} \label{yyx}
 y_{n+1}+y_n = -\frac{1-c}c b_nx_n + \alpha\beta - \frac{\gamma}{c} + \frac{1-c}{c} \sum_{k=0}^n x_k .
\end{gather}
Next we use \eqref{ABxy} in \eqref{AAB} and multiply everything by $(x+\alpha)(x+\beta)$. Again this gives a quadratic equation in $x$ in which the quadratic term vanishes. The linear term gives
\begin{gather} \label{difadify}
 \frac{1-c}{c} \big(a_{n+1}^2-a_n^2\big) = y_{n+1}-y_n +b_n+\alpha+\beta+n.
\end{gather}
Summing from $0$ to $n-1$ and using \eqref{bx} gives
\begin{gather} \label{aysumx}
 \frac{1-c}{c} a_n^2 = y_n + \sum_{k=0}^{n-1} x_k + \frac{n(n+\alpha+\beta-\gamma-1)}{1-c},
\end{gather}
which corresponds to \eqref{anur} in Corollary~\ref{cor1}. The constant term gives
\begin{gather} \label{difaxdify}
 \frac{1-c}{c} \big(a_{n+1}^2x_{n+1} -a_n^2 x_{n-1}\big) = -b_n(y_{n+1}-y_n) + \alpha\beta -y_n.
\end{gather}
Multiply this by $\frac{1-c}{c}x_n$ and use \eqref{yyx} to eliminate $b_nx_n$, then
\begin{gather*}
 \frac{(1-c)^2}{c^2} \big(a_{n+1}^2 x_{n+1}x_n-a_n^2x_nx_{n-1}\big)\\
 \qquad{} = y_{n+1}^2-y_n^2 - \left( \alpha\beta-\frac{\gamma}{c} \right) (y_{n+1}-y_n)
 - (y_{n+1}-y_n) \frac{1-c}{c} \sum_{k=0}^n x_k + \frac{1-c}{c} x_n( \alpha\beta - y_n) .
\end{gather*}
Summing and using summation by parts \eqref{sumpart} then gives
\begin{gather} \label{xxy}
\frac{(1-c)^2}{c^2} a_n^2 x_n x_{n-1} = y_n\left( y_n-\alpha\beta + \frac{\gamma}{c} \right)-(y_n-\alpha\beta)\frac{1-c}{c} \sum_{k=0}^{n-1}x_k.
\end{gather}
Finally, we use \eqref{ABxy} in \eqref{AABB} and multiply everything by $(x+\alpha)^2(x+\beta)^2$. This gives a cubic equation in~$x$ in which the cubic term vanishes. The quadratic term gives~\eqref{aysumx} again. The linear term gives
\begin{gather}
 \frac{(1-c)^2}{c^2} a_n^2 (x_n+x_{n-1}) = -y_n \left( n \frac{1+c}{c} + \alpha + \beta - \frac{\gamma+1}{c} \right) + \frac{(\alpha\beta-\gamma)n}{c} \nonumber\\
 \hphantom{\frac{(1-c)^2}{c^2} a_n^2 (x_n+x_{n-1}) =}{} + (\alpha+\beta+n) \frac{1-c}{c} \sum_{k=0}^{n-1} x_k, \label{xxyy}
\end{gather}
and the constant term gives \eqref{xxy} again.

So now we have equations \eqref{bx} and \eqref{aysumx} to express the recurrence coefficients $a_n^2$ and $b_n$ in terms of the sequences $(x_n,y_n)_n$. Furthermore we will use~\eqref{yyx},~\eqref{xxy} and~\eqref{xxyy} to find a~system of recurrence relations for $(x_n,y_n)_n$. Note that these three equations contain the sum $\sum\limits_{k=0}^{n-1} x_k$ which we will eliminate from these equations so that we are left with a system of two first order equations for $(x_n,y_n)_n$.

\begin{Theorem} \label{thm1} The sequences $(x_n,y_n)$ can be computed recursively using
\begin{gather}
\big( y_n-\alpha\beta+ (\alpha+\beta+n)x_n - x_n^2 \big) \big( y_{n+1}-\alpha\beta+(\alpha+\beta+n+1)x_n - x_n^2 \big) \nonumber\\
 \qquad{} = \frac{1}{c} ( x_n-1 ) ( x_n-\alpha ) ( x_n-\beta ) ( x_n-\gamma),\label{dP1}
\end{gather}
and
\begin{gather}
(x_n + Y_n)(x_{n-1} + Y_n) \nonumber\\
=\frac{(y_n+n\alpha) (y_n+n\beta) \bigl( y_n +n\gamma -(\gamma-\alpha)(\gamma-\beta) \bigr) \bigl( y_n + n - (1-\alpha)(1-\beta) \bigr)}
{\bigl( y_n(2n+\alpha+\beta-\gamma-1) + n\bigl( (n+\alpha+\beta)(n+\alpha+\beta-\gamma-1) -\alpha\beta + \gamma \bigr)\bigr)^2} ,\label{dP2}
\end{gather}
where
\begin{gather*} Y_n = \frac{y_n^2 + y_n \bigl( n(n+\alpha+\beta-\gamma-1) -\alpha\beta+\gamma \bigr) - \alpha\beta n(n+\alpha+\beta-\gamma-1)}
 {y_n(2n+\alpha+\beta-\gamma-1) + n\bigl( (n+\alpha+\beta)(n+\alpha+\beta-\gamma-1) -\alpha\beta + \gamma\bigr)}. \end{gather*}
The initial values are given by
\begin{gather*} y_0=0, \qquad
 x_0 = \frac{c\alpha\beta}{\gamma} \frac{{}_2F_1(\alpha+1,\beta+1;\gamma+1;c)}{{}_2F_1(\alpha,\beta;\gamma;c)} - \frac{(\alpha+\beta)c-\gamma}{1-c}. \end{gather*}
\end{Theorem}

\begin{proof}Multiply \eqref{yyx} by $y_n$ to find
\begin{gather} \label{bxy1}
 y_{n+1}y_n + y_n^2 = -\frac{1-c}{c} b_nx_ny_n + \alpha\beta y_n - \frac{\gamma}{c}y_n + y_n \frac{1-c}{c} S_n + \frac{1-c}{c}x_ny_n,
\end{gather}
where from now on we write
\begin{gather*} S_n = \sum_{k=0}^{n-1} x_k . \end{gather*}
Multiply \eqref{xxyy} by $x_n$ and use \eqref{bx} to find
\begin{gather*} \frac{1-c}{c} a_n^2\big(x_n^2+x_nx_{n-1}\big) = - b_nx_ny_n + x_n^2y_n + \frac{x_ny_n}{1-c} + \frac{(\alpha\beta-\gamma)nx_n}{1-c}
+ (\alpha+\beta+n)x_n S_n. \end{gather*}
Use \eqref{xxy} to remove $a_n^2 x_nx_{n-1}$ and then \eqref{aysumx} to remove the remaining $a_n^2$ to find
\begin{gather}
 x_n^2 \left( S_n + \frac{n(n+\alpha+\beta-\gamma-1)}{1-c}\right) + \frac{c}{1-c} y_n \left(y_n -\alpha\beta + \frac{\gamma}{c}\right)
 - (y_n-\alpha\beta) S_n \nonumber\\
\qquad{} = - b_nx_ny_n + \frac{x_ny_n}{1-c} + \frac{(\alpha\beta-\gamma)nx_n}{1-c} + (\alpha+\beta+n)x_nS_n.\label{bxy2}
\end{gather}
Eliminate $b_nx_ny_n$ from \eqref{bxy1} and \eqref{bxy2} to find
\begin{gather}
 y_{n+1}y_n - \frac{n(n+\alpha+\beta-\gamma-1)}{c} x_n^2 + x_ny_n + \frac{(\alpha\beta-\gamma)n x_n}{c} \nonumber\\
\qquad{} = \big( \alpha\beta - (\alpha+\beta+n)x_n + x_n^2 \big) \frac{1-c}{c} S_n.\label{temp}
\end{gather}
Now use \eqref{yyx} and \eqref{bx} to find
\begin{gather*} \frac{1-c}{c} S_n = y_{n+1}+y_n + \frac{1-c}{c}x_n \left( x_n + \frac{n+(n+ \alpha+\beta)c-\gamma}{1-c} \right) -\alpha\beta + \frac{\gamma}{c} - \frac{1-c}{c} x_n \end{gather*}
and use this to replace the sum $S_n$ in \eqref{temp}. This gives an expression containing $x_n$ and the terms $y_{n+1}+y_n$ and $y_ny_{n+1}$. Some calculus shows that it can be factored as in \eqref{dP1}.

Next, use \eqref{aysumx} to replace $a_n^2$ in \eqref{xxy} and \eqref{xxyy}. Then one can eliminate the sum $S_n = \sum\limits_{k=0}^{n-1} x_k$ from both equations and find
\begin{gather*}
 x_nx_{n-1} \bigl( y_n(2n+\alpha+\beta-\gamma-1) + n \bigl( (n+\alpha+\beta)(n+\alpha+\beta-\gamma-1)-\alpha\beta + \gamma\bigr) \bigr) \\
 \quad\quad{} + (x_n+x_{n-1}) \bigl( y_n^2 + y_n \bigl( n(n+\alpha+\beta-\gamma-1) -\alpha\beta+\gamma \bigr) - \alpha\beta n(n+\alpha+\beta-\gamma-1) \bigr) \\
\quad{} = y_n^2(-n+\gamma+1) + y_n (2\alpha\beta n + \alpha\gamma + \beta \gamma - \alpha\beta\gamma - \alpha\beta) - \alpha\beta(\alpha\beta-\gamma)n .
\end{gather*}
This is an equation containing $y_n$ and the sum $x_n+x_{n-1}$ and product $x_nx_{n-1}$. Some (lengthy) calculus shows it can be factored as in~\eqref{dP2}.

The initial values are $a_0^2=0$ and $b_0 = m_1/m_0$, where $m_0={}_2F_1(\alpha,\beta;\gamma;c)$ and $m_1= \frac{c\alpha\beta}{\gamma} {}_2F_1(\alpha+1,\beta+1;\gamma+1;c)$, which gives
\begin{gather*} y_0=0, \qquad x_0 = \frac{m_1}{m_0} - \frac{(\alpha+\beta)c-\gamma}{1-c}. \tag*{\qed} \end{gather*}\renewcommand{\qed}{}
\end{proof}

\looseness=-1 The system \eqref{dP1}--\eqref{dP2} is a system of two first order recurrence equations for $x_n$, $y_n$ and is a~discrete Painlev\'e equation, similar to d-P$\big(E_6^{(1)}/A_2^{(1)}\big)$ in \cite[Section~8.1.15]{KNY} or \cite[$E_6^{\delta}$ on p.~296]{GR}, except for a quadratic term $x_n^2$ on the left of~\eqref{dP1} and the rational term~$Y_n$ on the left of~\eqref{dP2}.

\section{The Toda lattice} \label{sec5}
If we put $c=c_0 e^t$ then the weight \eqref{hyperw} is an exponential modification of the weight with $c=c_0$ for $t=0$, and this deformation
(with deformation parameter $t$) corresponds to a Toda flow. The recurrence coefficients $a_n^2(t)$ and $b_n(t)$, as functions of the deformation parameter $t$, then satisfy the Toda equations
\begin{gather*}
 \frac{{\rm d}}{{\rm d}t} a_n^2 = a_n^2(b_n-b_{n-1}), \qquad n \geq 1, \\
 \frac{{\rm d}}{{\rm d}t} b_n = a_{n+1}^2 - a_n^2, \qquad n \geq 0,
\end{gather*}
see, e.g., \cite[Section~2.8]{Ismail} or \cite[Section~3.2.2]{WVA}. In the variable $c$, these Toda equations become
\begin{gather}
 c \frac{{\rm d}}{{\rm d}c} a_n^2= a_n^2(b_n-b_{n-1}), \qquad n \geq 1, \label{Toda1} \\
 c \frac{{\rm d}}{{\rm d}c} b_n= a_{n+1}^2 - a_n^2, \qquad n \geq 0. \label{Toda2}
\end{gather}
For the sequences $x_n$ and $y_n$ we then have:

\begin{Proposition} \label{prop2} The derivatives of $x_n$ and $y_n$ with respect to the parameter $c$ are given by
\begin{gather} \label{xder}
 x_n' = b_n' - \frac{2n+\alpha+\beta-\gamma}{(1-c)^2},
\end{gather}
and
\begin{gather} \label{yder}
 y_n' = - \frac{1+c}{c^2} a_n^2 + \frac{1-c}{c} \big(a_n^2\big)' ,
\end{gather}
where $'$ denotes derivation with respect to $c$. Furthermore the Toda equations for $(x_n,y_n)$ are
\begin{gather}
 (1-c) x_n'= y_{n+1}-y_n + x_n, \label{xToda} \\
 (1-c) y_n'= \frac{(1-c)^2}{c^2} a_n^2 (x_n-x_{n-1}), \qquad n \geq 0. \label{yToda}
\end{gather}
\end{Proposition}

\begin{proof} If we take the derivative with respect to $c$ in \eqref{bx} then we find \eqref{xder}. In a similar way we take the derivative in \eqref{aysumx} and \eqref{xder} to find
\begin{gather*} y_n' = - \frac{a_n^2}{c^2} + \frac{1-c}{c} \big(a_n^2\big)' - \sum_{k=0}^{n-1} b_k'. \end{gather*}
Now use the Toda equation \eqref{Toda2} to find
\begin{gather*} \sum_{k=0}^{n-1} b_k' = \frac{1}{c} \sum_{k=0}^{n-1} \big(a_{k+1}^2 -a_k^2\big) = \frac{a_n^2}{c}, \end{gather*}
where we used $a_0^2=0$. This gives \eqref{yder}. If we use \eqref{Toda2}, then \eqref{xder} becomes
\begin{gather*} x_n' = \frac{a_{n+1}^2-a_n^2}{c} - \frac{2n+\alpha+\beta-\gamma}{(1-c)^2}, \end{gather*}
which after using \eqref{aysumx} gives \eqref{xToda}. If we use \eqref{Toda1}, then \eqref{yder} becomes
\begin{gather*} y_n' = - \frac{1+c}{c^2} a_n^2 + \frac{1-c}{c^2} a_n^2 (b_n-b_{n-1}), \end{gather*}
and \eqref{bx} then gives \eqref{yToda}.
\end{proof}

\section{Painlev\'e VI} \label{sec6}

By combining the Toda equations \eqref{xToda}--\eqref{yToda} with the discrete Painlev\'e equations \eqref{dP1}--\eqref{dP2} one can in principle
find a differential equation for $x_n$ and $y_n$ as a function of the variable $c$, which after a suitable transformation can be reduced to
Painlev\'e VI. This approach is rather cumbersome and we were able to work it out by using computer algebra. However, we will present here another approach which gives an easier differential equation for $S_n = \sum\limits_{k=0}^{n-1} x_k$.

\begin{Theorem} \label{thm2} If we put $\sigma_n = (c-1)S_n+Kc+L$, with
\begin{gather*}
 K = \alpha\beta - \frac14 (\alpha+\beta+n)^2, \\
 L = \frac14 \bigl( (\alpha+\beta+\gamma+1)n + \alpha^2+\beta^2 - (\alpha+\beta)(\gamma+1)+2\gamma \bigr),
\end{gather*}
then $\sigma_n$ satisfies the Painlev\'e VI $\sigma$-equation
\begin{gather}
 \sigma_n' \bigl[ c(c-1) \sigma_n''\bigr]^2 + \bigl[ \sigma_n' \bigl(2\sigma_n - (2c-1) \sigma_n'\bigr) + d_1d_2d_3d_4\bigr]^2\nonumber \\
\qquad{} = \big(\sigma_n'+d_1^2\big)\big(\sigma_n'+d_2^2\big) (\sigma_n'+d_3^2\big)\big(\sigma_n'+d_4^2\big) , \label{sPVI}
\end{gather}
with parameters
\begin{gather*} d_1= \frac{n+\alpha-\beta}{2}, \!\!\qquad d_2=\frac{-n+\alpha-\beta}{2},\!\!\qquad d_3=\frac{n+\alpha+\beta-2}{2}, \!\!\qquad d_4= \frac{n+\alpha+\beta-2\gamma}{2}. \end{gather*}
\end{Theorem}

\begin{proof}Again we consider $S_n$ (and $\sigma_n$) as a function of the variable $c$ and derivatives are with respect to~$c$. Summing~\eqref{xToda} and using $y_0=0$ shows that
\begin{gather*} (1-c) S_n' = y_n + S_n, \end{gather*}
and hence
\begin{gather} \label{yS}
 y_n = [(1-c)S_n]', \qquad y_n' = [(1-c)S_n]''.
\end{gather}
Subtracting \eqref{xxyy} and \eqref{yToda} gives
\begin{gather}
 2 \frac{(1-c)^2}{c^2} a_n^2 x_n = (1-c)y_n' - y_n\left( n \frac{1+c}{c} + \alpha+\beta - \frac{\gamma+1}{c} \right) \nonumber\\
\hphantom{2 \frac{(1-c)^2}{c^2} a_n^2 x_n=}{} + \frac{(\alpha\beta-\gamma)n}{c} + (\alpha+\beta+n) \frac{1-c}{c} S_n,\label{xnyS}
\end{gather}
while adding \eqref{xxyy} and \eqref{yToda} gives
\begin{gather*}
 2 \frac{(1-c)^2}{c^2} a_n^2 x_{n-1}= -(1-c)y_n' - y_n\left( n \frac{1+c}{c} + \alpha+\beta - \frac{\gamma+1}{c} \right)\nonumber\\
 \hphantom{2 \frac{(1-c)^2}{c^2} a_n^2 x_{n-1}=}{} + \frac{(\alpha\beta-\gamma)n}{c}
 + (\alpha+\beta+n) \frac{1-c}{c} S_n.
\end{gather*}
If we multiply both expressions, then
\begin{gather}
4 \frac{(1-c)^4}{c^4} a_n^4 x_n x_{n-1} = \left[- y_n\left( n \frac{1+c}{c} + \alpha+\beta - \frac{\gamma+1}{c} \right)\right.\nonumber\\
\left.\hphantom{4 \frac{(1-c)^4}{c^4} a_n^4 x_n x_{n-1} =}{}+ \frac{(\alpha\beta-\gamma)n}{c} + (\alpha+\beta+n) \frac{1-c}{c} S_n \right]^2 - \bigl[(1-c) y_n'\bigr]^2.\label{xx1}
\end{gather}
Recall that \eqref{aysumx} gives
\begin{gather*} \frac{(1-c)^2}{c^2} a_n^2 = \frac{1-c}{c} (y_n+S_n) + \frac{n(n+\alpha+\beta-\gamma-1)}{c}, \end{gather*}
hence combining this with \eqref{xxy} gives
\begin{gather}
4 \frac{(1-c)^4}{c^4} a_n^4 x_n x_{n-1} = 4 \left( \frac{1-c}{c} (y_n+S_n) + \frac{n(n+\alpha+\beta-\gamma-1)}{c} \right)\nonumber \\
 \hphantom{4 \frac{(1-c)^4}{c^4} a_n^4 x_n x_{n-1} =}{} \times \left( y_n \left(y_n-\alpha\beta+\frac{\gamma}{c} \right) - (y_n-\alpha\beta) \frac{1-c}{c} S_n \right).\label{xx2}
\end{gather}
Clearly \eqref{xx1}--\eqref{xx2} gives the equation
\begin{gather*}
\left[- y_n\left( n \frac{1+c}{c} + \alpha+\beta - \frac{\gamma+1}{c} \right) + \frac{(\alpha\beta-\gamma)n}{c}
 + (\alpha+\beta+n) \frac{1-c}{c} S_n \right]^2 - \bigl[(1-c) y_n' \bigr]^2 \\
\qquad{} = 4 \left( \frac{1-c}{c} (y_n+S_n) + \frac{n(n+\alpha+\beta-\gamma-1)}{c} \right) \\
\qquad\quad{} \times \left( y_n \left(y_n-\alpha\beta+\frac{\gamma}{c} \right) - (y_n-\alpha\beta) \frac{1-c}{c} S_n \right),
\end{gather*}
and if we replace $y_n$ and $y_n'$ by~\eqref{yS}, then this is a non-linear second order differential equation for~$S_n$, or better for $\hat{\sigma}_n =(c-1)S_n$:
\begin{gather*}
 [c(1-c)\hat{\sigma}_n'']^2 + 4 \bigl( (c-1)\hat{\sigma}_n' - \hat{\sigma}_n + n(n+\alpha+\beta-\gamma-1) \bigr)\\
 \qquad\quad{}\times \bigl( \hat{\sigma}_n' (c\hat{\sigma}_n' +\alpha\beta c- \gamma)-
(\hat{\sigma}_n'+\alpha\beta)\hat{\sigma}_n \bigr) \\
 \qquad{} = \bigl[ \hat{\sigma}_n' \bigl( n+(\alpha+\beta+n)c-\gamma-1 \bigr) + (\alpha\beta-\gamma)n - (\alpha+\beta+n)\hat{\sigma}_n \bigr]^2.
\end{gather*}
If we now put $\sigma_n = \hat{\sigma}_n+Kc+L$, then a lengthy but straightforward computation gives the required Painlev\'e~VI $\sigma$-equation.
\end{proof}

The Painlev\'e $\sigma$-equations are given as equations $\textup{E}_{\scriptstyle\textrm{I}}$--$\textup{E}_{\scriptstyle\textrm{VI}}$ in \cite{Okamoto} and $\sigma$-equations $\sigma \textup{PII}$--$\sigma \textup{PVI}$ are given in \cite[equation~(8.15) in Section~8.1 or equation~(8.29) in Section~8.2]{Forrester}. They are equiva\-lent to the six Painlev\'e equations $\textup{P}_{\scriptstyle\textup{I}}$--$\textup{P}_{\scriptstyle\textup{VI}}$ in the sense that there is a one-to-one correspondence between solutions of the Painlev\'e equations and the corresponding $\sigma$-equations. Clearly our solution $\sigma_n$ is a special function solution which is expressed in terms of a Wronskian containing hypergeometric functions.

If $\sigma_n$ is known, then also $S_n$ is known, and then from \eqref{yS} it follows that $y_n = [(1-c)S_n]'$. Using \eqref{aysumx} we find
\begin{gather*} \frac{1-c}{c} a_n^2 = y_n + S_n +\frac{n(n+\alpha+\beta-\gamma-1)}{1-c}, \end{gather*}
so that the recurrence coefficient $a_n^2$ is in terms of $S_n$ and $S_n'$. For $x_n$ one can use \eqref{xnyS} to find that it is in terms of $S_n$, $S_n'$ and $S_n''$. Then finally~\eqref{bx} shows that~$b_n$ is also in terms of~$S_n$, $S_n'$ and $S_n''$. Hence $S_n$ and its first two derivatives are enough to find the quantities of interest for these discrete orthogonal polynomials. Furthermore $S_n''$ is in terms of $S_n$ and~$S_n'$ because of the $\sigma$-equation~\eqref{sPVI}.

\begin{remark*} Special function solutions of Painlev\'e VI are generated by a seed function that comes from a Riccati equation; see, e.g., \cite[Section~7.5]{ClarksonLN} or \cite[Section~6.2.5]{WVA}. We can show that $x_n$ indeed satisfies a Riccati equation and in particular that $x_0$ is the seed function. Differentiate~\eqref{xToda} with respect to~$c$ to find
\begin{gather*} (1-c)x_n'' = y_{n+1}' -y_n' + 2x_n' . \end{gather*}
Replace $y_{n+1}'$ and $y_n'$ by using the Toda equation \eqref{yToda}, then
\begin{gather} \label{xn2}
 (1-c)^2 x_n'' = \frac{(1-c)^2}{c^2} \bigl( a_{n+1}^2(x_{n+1}-x_n) - a_n^2(x_n-x_{n-1}) \bigr) + 2(1-c)x_n' .
\end{gather}
Combining \eqref{difadify} with \eqref{bx} and \eqref{xToda} gives
\begin{gather*} \frac{(1-c)^2}{c^2} \big(a_{n+1}^2-a_n^2\big) = \frac{(1-c)^2}{c} x_n' + \frac{2n+\alpha+\beta-\gamma}{c}. \end{gather*}
Multiply this by $x_n$ and add this to \eqref{xn2} to find
\begin{gather}
 (1-c)^2 x_n'' + \frac{(1-c)^2}{c} x_n x_n' + \frac{2n+\alpha+\beta-\gamma}{c} x_n \nonumber\\
\qquad{} = \frac{(1-c)^2}{c^2} \bigl( a_{n+1}^2 x_{n+1} + a_n^2 x_{n-1} - 2a_n^2 x_n \bigr) + 2(1-c)x_n' . \label{xn2-2}
\end{gather}
Multiply \eqref{difaxdify} by $\frac{1-c}{c}$ and add this to \eqref{xn2-2} to find
\begin{gather*}
 (1-c)^2 x_n'' + \frac{(1-c)^2}{c} x_nx_n' + \frac{2n+\alpha+\beta-\gamma}{c} x_n + \frac{1-c}{c} (y_n-\alpha\beta) \\
\qquad\quad{} + \frac{1-c}{c} \left( x_n+ \frac{n+(n+\alpha+\beta)c-\gamma}{1-c} \right) (y_{n+1}-y_n) \\
\qquad{} = 2 \frac{(1-c)^2}{c^2} a_n^2 (x_{n-1}-x_n) + 2(1-c) x_n' .
\end{gather*}
Use \eqref{xToda} and \eqref{yToda} to replace $y_{n+1}-y_n$ and $x_n-x_{n-1}$ and collect terms to find
\begin{gather}
c(1-c)x_n'' + 2 (1-c) x_nx_n' + \bigl( n+(n+\alpha+\beta-2)c-\gamma \bigr) x_n' - x_n^2+ (n+\alpha+\beta)x_n - \alpha\beta \nonumber\\
 \qquad{} = -y_n -2 c y_n'.\label{xn2yn}
\end{gather}
This equation contains only the functions $x_n''$, $x_n'$, $x_n$ and the functions $y_n'$, $y_n$. Observe that the left hand side can be written as
\begin{gather*} \bigl( c(1-c)x_n' + (1-c)x_n^2 +\bigl(n+(n+\alpha+\beta)c-\gamma-1\bigr)x_n - \alpha\beta c \bigr)' \end{gather*}
so that \eqref{xn2yn} is in fact a Riccati equation for $x_n$ if $y_n$ is given. Recall that $y_0=0$, therefore
we have the Riccati equation for $x_0$
\begin{gather*} c(1-c)x_0' + (1-c)x_0^2 +\bigl((\alpha+\beta)c-\gamma-1\bigr)x_0 - \alpha\beta c = \textrm{constant}. \end{gather*}
One can verify this, using the fact that
\begin{gather*} x_0(c) = \frac{c\alpha\beta}{\gamma} \frac{{}_2F_1(\alpha+1,\beta+1;\gamma+1;c)}{{}_2F_1(\alpha,\beta;\gamma;c)}
- \frac{(\alpha+\beta)c-\gamma}{1-c}, \end{gather*}
and it turns out that the constant is~$-\gamma$. This Riccati equation for $x_0$ gives a seed function for all the special function solutions~$x_n$.
\end{remark*}

\section{Asymptotic behavior} \label{sec4}

As we mentioned before, the case $\alpha=\gamma$ (or $\beta=\gamma$) gives the Meixner polynomials for which the recurrence coefficients
are known, see \eqref{Meixner}. If we compare this with \eqref{bx} and \eqref{aysumx}, then it follows that for this special case
\begin{gather*} x_n = \gamma, \qquad y_n =-n\gamma. \end{gather*}
The sequence $(x_n)_n$ is a constant sequence and the constant is a zero of the right hand side of~\eqref{dP1}.

In Fig.~\ref{Fig1} we computed the $(x_n,y_n)_n$ for $\alpha=3/2$, $\beta=3$, $\gamma=1/3$ and $c=1/2$ for the weights~$w_k$ in~\eqref{weights} on the integers $\mathbb{N} = \{0,1,2,3,\ldots\}$ by using the recurrence \eqref{dP1}--\eqref{dP2}. We used a precision of~\texttt{Digits:=50} in Maple, because for \texttt{Digits:=10} the resulting values for~$x_n$,~$y_n$ were wrong when $n \geq 40$. The precision \texttt{Digits:=20} gives the same plots and the computed values only go wrong for $n \geq 80$. The initial values are
\begin{gather*} y_0=0, \qquad x_0 = \frac{m_1}{m_0} - \frac{(\alpha+\beta)c-\gamma}{1-c}, \end{gather*}
where
\begin{gather*} m_0= {}_2F_1(\alpha,\beta;\gamma;c), \qquad m_1 = \frac{c\alpha\beta}{\gamma} {}_2F_1(\alpha+1,\beta+1;\gamma+1;c). \end{gather*}
The calculations seem to suggest that~$x_n$ converges to~$\gamma$ and that~$y_n$ decreases linearly for large~$n$.

\begin{figure}[ht]\centering
\includegraphics[width=2.8in]{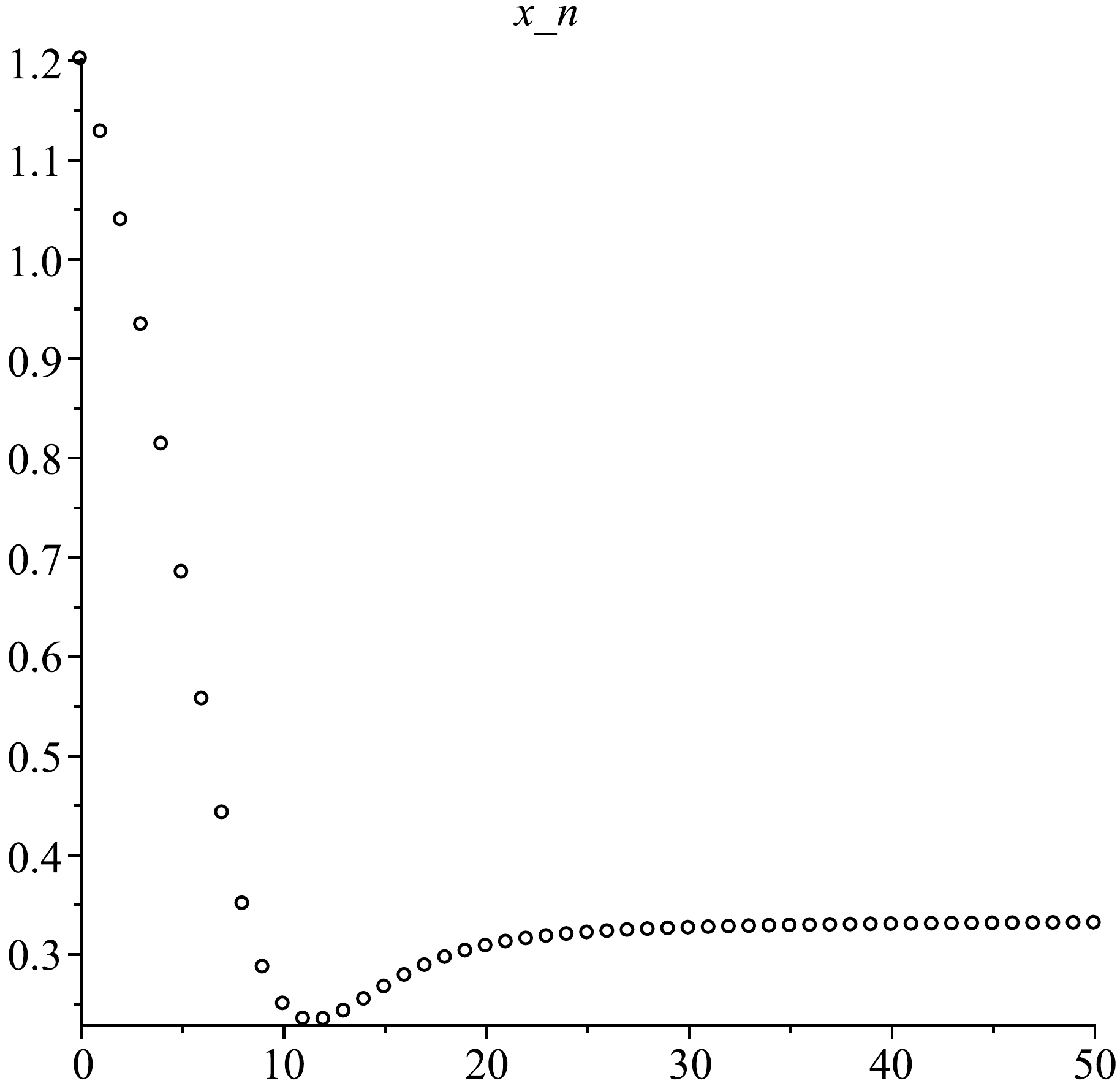} \qquad
\includegraphics[width=2.8in]{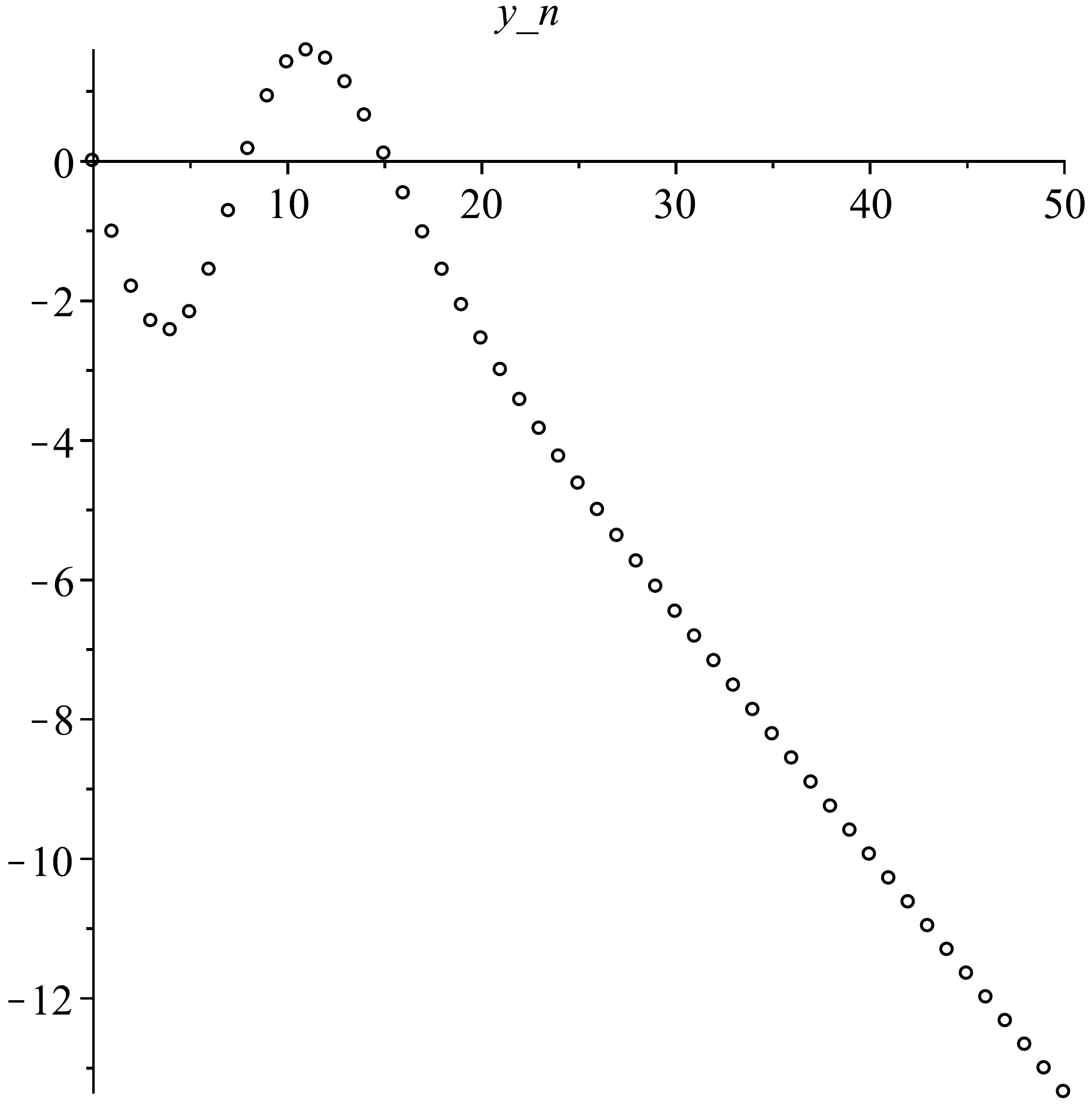}
\caption{The sequences $x_n$ (left) and $y_n$ (right) for $(\alpha,\beta,\gamma,c) = \big(\frac32,3,\frac13,\frac12\big)$.}\label{Fig1}
\end{figure}

We conjecture that for this initial value for $x_0$
\begin{gather*} \lim_{n \to \infty} x_n = \gamma \end{gather*}
and then \eqref{dP1} implies that
\begin{gather*} \lim_{n \to \infty} ( y_n +n\gamma ) = (\gamma-\alpha)(\gamma-\beta). \end{gather*}
Note that the latter is a zero of the right hand side of~\eqref{dP2}. This asymptotic behavior seems to be confirmed by the numerical results. We believe that this is the only initial value which gives this asymptotic behavior. The calculations are very sensitive of the initial value: a slight change of $x_0$ gives a more erratic behavior of the $(x_n,y_n)$ for large $n$.

We can also use the weights \eqref{weights} on the shifted lattice $\mathbb{N}+1-\gamma$ since $w(-\gamma)=0$ for the function $w$ in~\eqref{hyperw}, in a similar way as was done in \cite[Sections~2.2 and~3.2]{Smet}. The same recurrence relations still hold for the corresponding recurrence coefficients $\hat{a}_n,\hat{b}_n$. These recurrence coefficients are related to the recurrence coefficients of the weights on $\mathbb{N}$ by
\begin{gather*}
 \hat{a}_n^2(\alpha,\beta,\gamma,c) = a_n^2(\alpha-\gamma+1,\beta-\gamma+1,2-\gamma,c), \\
 \hat{b}_n(\alpha,\beta,\gamma,c) = b_n(\alpha-\gamma+1,\beta-\gamma+1,2-\gamma,c) +1-\gamma.
\end{gather*}
We conjecture that in this case
\begin{gather*} \lim_{n \to \infty} \hat{x}_n = 1 \end{gather*}
and then \eqref{dP1} gives
\begin{gather*} \lim_{n \to \infty} (\hat{y}_n + n) = (1-\alpha)(1-\beta). \end{gather*}
The corresponding initial values are
\begin{gather*} \hat{y}_0=0, \qquad \hat{x}_0 = x_0(\alpha-\gamma+1,\beta-\gamma+1,2-\gamma,c) +\gamma -1 . \end{gather*}
Again we believe this is the only initial value for which this asymptotic behavior holds. An interesting question is to find out which initial values give an asymptotic behavior of the form
\begin{gather*} \lim_{n\to \infty} x_n = \alpha, \qquad \textrm{or} \qquad \lim_{n \to \infty} x_n = \beta, \end{gather*}
which are the other two zeros of the right hand side of~\eqref{dP1}.

\section{Concluding remarks}

The reason why we considered the hypergeometric weights \eqref{weights} in this paper is twofold. On one hand they generalize various other discrete weights that were already analyzed in the li\-te\-rature (Charlier, Meixner, generalized Charlier, generalized Meixner, generalized Krawtchouk), building up from explicit rational expressions for the recurrence coefficients to second order non-linear recurrence and differential equations (Painlev\'e III and Painlev\'e V). On the other hand, it was already known that Painlev\'e VI has special function solutions in terms of Wronskians with hypergeometric functions \cite[Section~6.2.5]{WVA}, and such Wronskians appear naturally in formulas for the recurrence coefficients $a_n^2$ and $b_n$ for orthogonal polynomials:
\begin{gather*} a_n^2 = \frac{\Delta_{n+1}\Delta_{n-1}}{\Delta_n^2}, \qquad b_n = \frac{\Delta_{n+1}^*}{\Delta_{n+1}} - \frac{\Delta_n^*}{\Delta_n}, \end{gather*}
where $\Delta_n = \det ( m_{i+j} )_{i,j=0}^{n-1}$ is the Hankel determinant and $\Delta_n^*$ is obtained from $\Delta_n$ by repla\-cing the
last column $(m_{n-1},m_n,\ldots,m_{2n-2})^{\rm T}$ by $(m_n,m_{n+1},\ldots,m_{2n-1})^{\rm T}$ and $(m_n)_{n \in \mathbb{N}}$ are the moments
\begin{gather*} m_n = \sum_{k=0}^\infty k^n w_k . \end{gather*}
The moment $m_0$ is a Gauss hypergeometric series and all the other moments can be obtained from them by
\begin{gather*} m_n = \left( c \frac{{\rm d}}{{\rm d}c} \right)^n m_0, \end{gather*}
so that $\Delta_n$ and $\Delta_n^*$ are Wronskians. The challenge was to find the discrete Painlev\'e equations (Theorem~\ref{thm1}) and the continuous Painlev\'e equation (Theorem \ref{thm2}) for the recurrence coefficients of the orthogonal polynomials with these hypergeometric weights, using only standard properties of orthogonal polynomials and a number of suitable transformations. The system of discrete Painlev\'e equations \eqref{dP1}--\eqref{dP2} seems to be new but closely related to d-P$\big(E_6^{(1)}/A_2^{(1)}\big)$, and the Painlev\'e equation~\eqref{sPVI} is the $\sigma$-form of the Painlev\'e VI equation.

Note that one can write the weights in \eqref{weights} as
\begin{gather*} w_k = \frac{\Gamma(\gamma)}{\Gamma(\alpha)\Gamma(\beta)} \frac{\Gamma(\alpha+k) \Gamma(\beta+k)}{\Gamma(\gamma+k) \Gamma(k+1)} c^k := w(k) \end{gather*}
and that $w(-1)=0$, which gives a boundary condition for the weights on the lattice $\mathbb{N}=\{0,1,2,3,\ldots \}$. One also has
$w(-\gamma)=0$ so that one can also use these weights on the shifted lattice $\mathbb{N}+1-\gamma = \{1-\gamma, 2-\gamma,3-\gamma,\ldots\}$, as was done for generalized Charlier and Meixner polynomials in \cite{Smet}. The recurrence coefficients will satisfy the same discrete Painlev\'e equations but with a different initial value because $m_0$ and $m_1$ are different hypergeometric series: one has $a_0^2=0$ and $b_0 = m_1/m_0$ with
\begin{gather*}
 m_0 = \frac{\Gamma(\gamma)}{\Gamma(\alpha)\gamma(\beta)} \sum_{k=0}^\infty \frac{\Gamma(\alpha+k+1-\gamma)\Gamma(\beta+k+1-\gamma)}
 {\Gamma(k+1) \Gamma(k+2-\gamma)} c^{k+1-\gamma} \\
 \hphantom{m_0}{}= \frac{\Gamma(\gamma)\Gamma(\alpha+1-\gamma)\Gamma(\beta+1-\gamma)}{\Gamma(2-\gamma)\Gamma(\alpha)\Gamma(\beta)}
 c^{1-\gamma} {}_2F_1(\alpha+1-\gamma,\beta+1-\gamma;2-\gamma;c),
\end{gather*}
and $m_1 = c \frac{{\rm d}}{{\rm d}c} m_0$. The same Painlev\'e VI $\sigma$-equation will hold, but the solution will be a~different special function solution coming from a different hypergeometric seed function. The general case is obtained by taking a combination of both lattices, giving a one parameter family of special function solutions of Painlev\'e VI.

Note that nonlinear recurrence relations for the recurrence coefficients of these orthogonal polynomials were found by Dominici in \cite[Theorem~4]{Domin}, but these were of higher order and were not identified as discrete Painlev\'e equations. Our version \eqref{dP1}--\eqref{dP2} has the advantage that one can predict the asymptotic behavior of $a_n^2$ and $b_n$ (or $x_n$ and $y_n$) as $n \to \infty$ from them, and in Section~\ref{sec4} we conjectured this asymptotic behavior when the weights are on the lattice~$\mathbb{N}$ and on the shifted lattice $\mathbb{N}+1-\gamma$.

\subsection*{Acknowledgements}
GF acknowledges the support of the National Science Center (Poland) via grant OPUS 2017/25/ B/BST1/00931. Support of the Alexander von Humboldt Foundation is also greatfully acknow\-ledged. WVA is supported by FWO research project G.0864.16N and EOS project PRIMA 30889451. The authors thank the anonymous referees for their comments, which improved the original version.

\pdfbookmark[1]{References}{ref}
\LastPageEnding

\end{document}